\documentclass{article}
\usepackage{amsmath,amssymb,amsthm}
\usepackage{pstricks,pst-node}
\usepackage{graphicx}

\newtheorem{thm}{Theorem} 

\newtheorem{lem}[thm]{Lemma}

\newcommand{\td}{\operatorname{td}}

\title{Minimal obstructions for tree-depth:\\ A non-1-unique example}

\author{Michael D. Barrus\thanks{Department of Mathematics, University of Rhode Island, Kingston, Rhode Island 02881, United States; email: \texttt{barrus@uri.edu}} ~and
John Sinkovic\thanks{Department of Combinatorics and Optimization, University of Waterloo, Waterloo, Ontario, N2L 2G1, Canada; email: \texttt{johnsinkovic@gmail.com}}}

\begin{document}
\maketitle

\begin{abstract}
A $k$-ranking of a graph $G$ is a labeling of the vertices of $G$ with values from $\{1,\dots,k\}$ such that any path joining two vertices with the same label contains a vertex having a higher label. The tree-depth of $G$ is the smallest value of $k$ for which a $k$-ranking of $G$ exists. The graph $G$ is $k$-critical if it has tree-depth $k$ and every proper minor of $G$ has smaller tree-depth.

As defined in~\cite{BarrusSinkovic15}, a graph $G$ is 1-unique if for every vertex $v$ in $G$, there exists an optimal ranking of $G$ in which $v$ is the unique vertex with label 1. In \cite{BarrusSinkovic15} and \cite{unpub} the authors showed that several classes of critical graphs are $1$-unique and asked whether all critical graphs have this property. We answer in the negative by demonstrating an infinite family of graphs which are critical but not 1-unique.

\medskip
\emph{Keywords:} Graph minors, tree-depth, vertex ranking
\end{abstract}

\section{Introduction}
The \emph{tree-depth} of a graph $G$, denoted $\td(G)$, is defined as the smallest natural number $k$ such that the vertices of $G$ may be labeled with elements of $\{1,\dots,k\}$ such that every path joining two vertices with the same label contains a vertex having a larger label.  Define a graph $G$ to be \emph{critical} if every proper minor of $G$ has tree-depth less than $\td(G )$.  A critical graph with tree-depth $k$ will be called $k$-critical. 

Given a graph $G$, we will call a labeling of the vertices of $G$ with labels from $\{1,\dots,k\}$ a \emph{feasible} labeling if every path in $G$ between two vertices with the same label ($\ell$, say) passes through a vertex with a label greater than $\ell$. Adopting terminology from previous authors, we call a feasible labeling with labels from $\{1,\dots,k\}$ a \emph{($k$-)ranking of $G$}, and we refer to the labels as \emph{ranks} or \emph{colors} (note that every feasible labeling is a proper coloring of $G$). We call a ranking of $G$ \emph{optimal} if it is a $\td(G)$-ranking.  A graph $G$ is \emph{$1$-unique} if for every vertex $v$ of $G$ there is an optimal ranking of $G$ in which vertex $v$ is the only vertex receiving rank 1.

In the papers \cite{BarrusSinkovic15} and \cite{unpub}, the authors present many examples of $k$-critical graphs and show that they are all 1-unique, conjecturing that all critical graphs are 1-unique.  After further investigation, this turns out not to be the case; the purpose of this note is to demonstrate an infinite family of critical graphs, each of which is not 1-unique. In the next section we begin with some necessary background and results and then present the details of the counterexample.

\section{An infinite family of critical graphs that are not 1-unique}

We begin by recalling a definition and result from~\cite{BarrusSinkovic15}. 

Given a vertex $v$ in a graph $G$, a \emph{star-clique transform at $v$} removes $v$ from $G$ and adds edges between the vertices in $N_G(v)$ so as to make them a clique.  

\begin{thm}\label{thm:starclique}(\cite{BarrusSinkovic15}) Let $v$ be a vertex of a graph $G$, and let $H$ be the graph obtained through the star-clique transform at $v$ of $G$.  Vertex $v$ is $1$-unique in $G$ if and only if $\td(H)<\td(G)$.
\end{thm}

%

For any positive integer $k$, define a \emph{$k$-net} to be the graph constructed by attaching a single pendant vertex to each vertex of the complete graph $K_k$. The following fact is a special case of Lemma~2.7 in~\cite{unpub}.

\begin{lem}\label{lem:k-net}
A $k$-net has tree-depth $k+1$.
\end{lem}

We now present one more result on tree-depth before describing the family of counterexamples.

\begin{lem}\label{lem: Ka Box K2}
For any positive integer $a$, the graph $K_a \Box K_2$ has tree-depth $\lceil 3a/2\rceil$.
\end{lem}
\begin{proof}
The claim is easily verified for $a \in \{1,2\}$, so suppose $a \geq 3$. Let $V_1$ and $V_2$ denote the disjoint vertex sets of the two induced copies of $K_a$ in $K_a \Box K_2$. We may group the vertices of $K_a \Box K_2$ into pairs $\{u,u'\}$, where $uu'$ is an edge and $u$ and $u'$ are elements of $V_1$ and $V_2$, respectively. Any cutset $T$ in $K_a \Box K_2$ must contain at least one vertex from each such pair, so $|T| \geq a$. Moreover, since $K_a \Box K_2$ has independence number 2, after deleting any cutset the resulting graph has exactly two components, which must be complete subgraphs, the larger of which has at least $a - |T|/2$ vertices. It follows that the tree-depth of $K_a \Box K_2$ is at least $a + |T|/2$, which is at least $\lceil 3a/2\rceil$.

To demonstrate equality, let $T$ be a subset of $V(K_a \Box K_2)$ consisting of $\lfloor a/2 \rfloor$ vertices from $V_1$ and $\lceil a/2 \rceil$ vertices from $V_2$, with no vertex in $T \cap V_1$ adjacent to any vertex in $T \cap V_2$. Label the vertices in $T$ injectively with labels from $\{\lceil a/2 \rceil+1, \dots, \lceil 3a/2\rceil\}$, and in each of $V_1 - T$ and $V_2 - T$, injectively label the vertices with $\{1,\dots,\lceil a/2 \rceil\}$. It is straightforward to verify that this is a feasible labeling using the appropriate number of colors.
\end{proof}

\begin{thm}
For any $n \geq 4$, let $H_n$ be the graph obtained by subdividing (once) all edges incident with a given vertex $v$ of $K_{n}$. The graph $H_n$ is $(n+1)$-critical but not $1$-unique; moreover, $v$ is the only non-1-unique vertex in $H_n$.
\end{thm}
\begin{proof}
In the following, let $A_n$ denote the vertices of degree $2$ incident with $v$ in $H_n$, and let $B_n$ denote $V(H_n)-v-A_n$; note that the $n-1$ vertices in $B_n$ form a clique in $H_n$, and each vertex in $B_n$ is adjacent exactly to the other vertices of $B_n$ and to a single vertex in $A_n$ (with each vertex in $A_n$ having a single neighbor in $B_n$).

To see that $\td(H_n) \leq n+1$, injectively label the vertices of $B_n$ with labels $2,\dots,n$, label each vertex in $A_n$ with $1$, and label $v$ with $n+1$. Under this labeling only vertices in $A_n$ receive a common label, and each path joining two vertices in $A_n$ contains a vertex outside $A_n$, which has a higher label than $1$.

For convenience in proving that $\td(H_n) \geq n+1$, we now construct a graph $H_3$ in the same way that $H_n$ is defined for $n \geq 4$; note that $H_3$ is isomorphic to $C_5$. By induction we show that $\td(H_n) \geq n+1$ for all $n \geq 3$). 

Observe that $H_3$ has tree-depth $4$, as desired. Now suppose that for some integer $k \geq 3$ we have $\td(H_k) \geq k+1$. Now consider the result of deleting a vertex from $H_{k+1}$. If the vertex deleted is $v$, the remaining graph is isomorphic to a $k$-net, which by Lemma~\ref{lem:k-net} has tree-depth $k+1$. Deleting any vertex from $A_{k+1}$ or from $B_{k+1}$, along with its neighbor in the other set, leaves a copy of $H_k$, which by our induction hypothesis has tree-depth $k+1$. Thus $\td(H_{k+1}) \geq 1 + \td(H_k) \geq (k+1)+1$, as desired.

We now show that $H_n$ is critical. Note that if $u$ is any vertex in $A_n$, and if $w$ is the neighbor of $u$ in $B_n$, then each of $H_n - uv$ and $H_n-uw$ may be feasibly colored by labeling $v$ and $w$ with $2$, labeling all of $A_n$ with $1$, and injectively labeling the vertices of $B_n - u$ with colors from $3,\dots,n$.

If $w,w'$ are vertices in $B_n$, we feasibly color $H_n - ww'$ by labelling $w$ and $w'$ with $1$, labeling all vertices in $A_n$ with $2$, labeling $v$ with $3$, and injectively labeling the vertices of $B_n - \{w,w'\}$ with colors from $4,\dots,n$.

Contracting an edge of $H_n$ that is incident with a vertex in $A_n$ yields a graph isomorphic to that obtained by adding to $H_{n-1}$ a vertex $w'$ adjacent to the analogous vertex $v$ and to all vertices of $B_{n-1}$; we feasibly color this graph by labeling $w'$ and all vertices in $A_{n-1}$ with $1$, labeling $v$ with $2$, and injectively labeling vertices in $B_{n-1}$ with $\{3,\dots,n\}$.

Contracting an edge $w_1w_2$ of $H_n$, where $w_1,w_2 \in B_n$, yields a graph that can be feasibly colored in the following way: label all vertices of $A_n$ with 1, label the vertex replacing $w_1$ and $w_2$ with 2, label $v$ with 3, and injectively label the vertices of $B_n-\{w_1,w_2\}$ with $\{4,\dots,n\}$. Having shown now that deleting or contracting any edge results in a graph with smaller tree-depth (and, it follows, the same holds if we delete any vertex), we conclude that $G$ is $(n+1)$-critical.

Now by Theorem~\ref{thm:starclique}, $v$ will be 1-unique if and only if performing a star-clique transformation at  $v$ in $H_n$ yields a graph with a lower tree-depth. Observe that a star-clique transformation on $v$ actually yields $K_{n-1} \Box K_2$. By Lemma~\ref{lem: Ka Box K2}, $\td(K_{n-1}\Box K_2) = \lceil \frac{3}{2}(n-1)\rceil$, which is at least $n+1$ for $n \geq 4$, rendering $v$ non-1-unique. Note, however, that a star-clique transformation on any vertex of $A_n$ or $B_n$ has the same effect as contracting an edge between $A_n$ and $B_n$ in $H_n$, which lowers the tree-depth, as we verified above; hence all vertices of $H_n$ other than $v$ are 1-unique.
\end{proof}


\bibliography{treedepth}
\bibliographystyle{elsarticle-num}

\end{document}